\documentclass[a4paper,11pt,seceq]{article}

\usepackage{amsmath}
\usepackage{mathtools}
\usepackage{bm}
\usepackage{wasysym}
\usepackage{float}
\usepackage{authblk}
\usepackage{subcaption}
\usepackage{graphicx}

\usepackage{sectsty}
\sectionfont{\fontsize{13}{13}\selectfont}
\subsectionfont{\fontsize{12}{12}\selectfont}
\subsubsectionfont{\fontsize{12}{12}\selectfont}

\numberwithin{equation}{subsection}

\usepackage{geometry}
\usepackage{amsthm}
\newtheorem{theorem}{Theorem}

\usepackage{setspace}
\newcommand{\hypgeo}[2]{%
  {\vphantom{F}}_{#1}\kern-\scriptspace F_{#2}%
}

\allowdisplaybreaks

\begin{document}

\begin{flushleft}
    {\Large\bf Stochastic Models of Stem Cells and Their Descendants under Different Criticality Assumptions}
\end{flushleft}
 
\begin{center}
    Nam H Nguyen\footnote{CONTACT: hn17@rice.edu} and Marek Kimmel
\end{center}
\begin{center}
    Department of Statistics, Rice University, Houston, TX, USA
\end{center}
  
\begin{abstract}
    \noindent We study time continuous branching processes with exponentially distributed lifetimes, with two types of cells that proliferate according to binary fission. A range of possible system dynamics are considered, each of which is characterized by the mutation rate of the original cells and the survival probability of the altered cells' progeny. For each system, we derive a closed-form expression for the joint probability generating function of cell counts, and perform asymptotic analysis on the behaviors of the cell population with particular focus on probability of extinction. Part of our results confirms known properties of branching processes using a different approach while other are original. While the model is best suited for modeling the fate of differentiating stem cells, we discuss other scenarios in which these system dynamics may be applicable in real life. We also discuss the history of the subject.
\end{abstract}  


\newpage

%%%%%%%%%%%%%%%%%%%%%%%%%%%%%%%%%%%%%%%%%%%%%%%%%%%%%%%%%%%%%%%%%%%%%%%%%%%%%%%%%%%%%%%%%%%%%%%

\section{Introduction}

All cells existing in muticellular organisms originate from a single fertilized egg. In embryonic development, the descendants of this single cell migrate and differentiate to create diverse tissues and organs. Even in a mature organism, the functional cells which perform a range of tasks have to renew, since they are ``used up", and their building materials such as nucleic acids, proteins and others are recycled or disposed of. In most cases, this self-renewal system has a hierarchical structure, with the self-renewing stem cells ($A$-cells) at the top, which divide into two progeny, each of which may remain a stem cell or change (differentiate) into a commited cell ($B$-cell). It is known that stem cells are highly protected and that stem cells usually have exactly 2 progeny. In addition to this, under physiological conditions, the number of stem $A$-cells remains roughly constant, until the organism begins to age. The commitment is irreversible, and commited $B$-cells proliferate so that a constant flux of their descendants is ensured. The count of the committed $B$-cells remains roughly constant, too. In the presence of a constant flux of differentiated $A$-cells, this is only possible if the fraction of $B$-cells proceeding to self-renewal after division is less than $1$. If it is greater than or equal to $1$, then the system increases in volume (cell count), starting from the $B$-cells. Therefore, the basic system that is interesting under physiological conditions is, in the branching process language, a critical process of ${A}$-cells feeding into the sub-critical process of ${B}$-cells. A sufficient condition  for this to happen is that each $A$-cell progeny becomes (on average) a $B$-cell with probability equal to $1/2$, and each $B$-cell progeny remains (on average) a $B$-cell with probability less than $1/2$.

However, in the early development phase, during aging, or under other conditions in which differentiation of $B$-cells is altered, the sub-criticality of $B$-cells may not be ensured. Less likely, but not inconceivably, the $A$-cells might be sub- or super-critical (albeit just slightly). Hence, we will consider a range of possibilities. The process can be described as follows
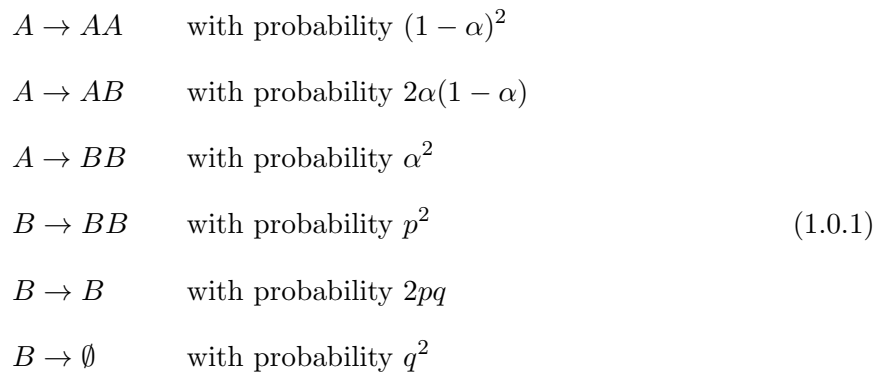
\begin{alignat}{2}\label{dynamics}
    &A\rightarrow AA\qquad &&\text{with probability}\ (1-\alpha)^2\nonumber\\
    &A\rightarrow AB\qquad &&\text{with probability}\ 2\alpha(1-\alpha)\nonumber\\
    &A\rightarrow BB\qquad &&\text{with probability}\ \alpha^2\nonumber\\
    &B\rightarrow BB\qquad &&\text{with probability}\ p^2\\
    &B\rightarrow B\qquad &&\text{with probability}\ 2pq\nonumber\\
    &B\rightarrow \emptyset\qquad &&\text{with probability}\ q^2\nonumber
\end{alignat}
where $\alpha\in[0,1]$, $p\in[0,1]$, and $q=1-p$. In addition, we assume that the lifetimes of $A$-cells and $B$-cells are exponentially distributed with parameters $\lambda_A$ and $\lambda_B$ respectively.

The paper is organized as follows. In Section 2, we introduce notations that will be used for the rest of the paper. In Section 3, we mention useful properties of special functions, which will be used in the mathematical derivations. In Section 4.1, we perform exploratory analysis to understand how cell counts behave on average over time. Section 4.2 contains the main text of the paper, where we derive closed-form expressions for the joint probability generating functions that fully describe the population dynamics at any given time. Based on the explicit results, we study the asymptotic extinction probabilities of the system under different combinations of criticalities in Section 4.3. In Section 5, we  analyze the results and provide literature examples of diseases which might be, perhaps metaphorically, due to all types of deviations from the critical $A$-cells, sub-critical $B$-cells stereotype. Section 6 is the conclusion of the paper. 

%%%%%%%%%%%%%%%%%%%%%%%%%%%%%%%%%%%%%%%%%%%%%%%%%%%%%%%%%%%%%%%%%%%%%%%%%%%%%%%%%%%%%%%%%%%%%%%

\section{Notations} 
Let $F_A(x,y,t)$ be the joint probability generating function (PGF) for the number of cells at time $t$ given that the process is initiated by a single cell of type $A$. That is, let $Z_A(t)$ and $Z_B(t)$ be the number of $A$-cells and $B$-cells at time $t$ respectively. Then, $F_A(x,y,t)$ is defined as follows
\begin{align*}
    F_A(x,y,t)=E\big[x^{Z_A(t)}y^{Z_B(t)}\big]
\end{align*}
Since we start with one single cell of type $A$, the initial condition for $F_A$ is $F_A(x,y,0)=x$. Similarly, $F_B(x,y,t)$ is the joint PGF given that the process is initiated by a single cell of type $B$. Note that cells of type $A$ mutate irreversibly to cells of type $B$. Hence, $F_B(x,y,t)$ is independent of $x$, and hence can be written more succintly as $F_B(y,t)$. For consistency of notation, we will write $F_B(x,y,t)$ for the rest of the paper. For $F_B$, we start with one cell of type $B$, so the initial condition is $F_B(x,y,0)=y$.

\noindent Let $h_A(x,y)$ and $h_B(x,y)$ be the progeny PGFs for cells of type $A$ and cells of type $B$ respectively. From (\ref{dynamics}), it follows that
\begin{equation}\label{eq:progeny_A}
    h_A(x,y)=[(1-\alpha)x+\alpha y]^2
\end{equation}
\begin{equation}\label{eq:progeny_B}
    h_B(x,y)=(p+qy)^2
\end{equation}

%%%%%%%%%%%%%%%%%%%%%%%%%%%%%%%%%%%%%%%%%%%%%%%%%%%%%%%%%%%%%%%%%%%%%%%%%%%%%%%%%%%%%%%%%%%%%%%

\section{Special functions}

\noindent Throughout the paper, we will make use of various special functions. Let $I(a,z)$ be the modified Bessel function of order $a$ evaluated at $z$. The first derivative of $I(a,z)$ is given by
\begin{equation}\label{Bessel_derivative}
    \frac{d}{dz}I(a,z)=\frac{1}{2}(I(a-1,z)+I(a+1,z))
\end{equation}
The following recursive formulae is also useful to simplify the results
\begin{equation}\label{Bessel_recursive}
    I(a-1,z)-I(a+1,z)=\frac{2a}{z}I(a,z)
\end{equation}
For asymptotic analysis, we need the following large-$z$ approximation of $I(a,z)$
\begin{equation}\label{Bessel_asymptotic}
    I(a,z)\sim\frac{1}{\sqrt{2\pi z}}e^z
\end{equation}
Let $M(a,b,z)$ and $W(a,b,z)$ be the Whittaker functions (first kind and second kind, respectively) with parameters $a$ and $b$ evaluated at $z$. The first derivatives are given by
\begin{align}\label{Whittaker_derivative}
    &M'(a,b,z)=\Big(\frac{1}{2}-\frac{a}{z}\Big)M(a,b,z)+\frac{\frac{1}{2}+a+b}{z}M(a+1,b,z)\nonumber\\
    &W'(a,b,z)=\Big(\frac{1}{2}-\frac{a}{z}\Big)W(a,b,z)-\frac{1}{z}W(a+1,b,z)
\end{align}
Let $\hypgeo{1}{1}(a,b,z)$ and $U(a,b,z)$ denote the confluent hypergeometric functions (first kind and second kind, respectively) with parameter $a$ and $b$ evaluated at $z$. There is a simple algebraic relationship between the Whittaker functions and the confluent hypergeometric functions
\begin{align}\label{Whittaker_Hypergeometric}
    &M(a,b,z)=e^{-z/2}z^{a+1/2}\hypgeo{1}{1}\bigg(b-a+\frac{1}{2},1+2b,z\bigg)\nonumber\\
    &W(a,b,z)=e^{-z/2}z^{b+1/2}U\bigg(b-a+\frac{1}{2},1+2b,z\bigg)
\end{align}
The large-$z$ behavior of $\hypgeo{1}{1}(a,b,z)$ and $U(a,b,z)$ is given by
\begin{align}\label{Hypergeometric_asymptotic}
    &\hypgeo{1}{1}(a,b,z)=\frac{\Gamma(b)}{\Gamma(a)}e^zz^{a-b}(1+O(|z|^{-1}))\nonumber\\
    &U(a,b,z)=z^{-a}(1+O(|z|^{-1}))
\end{align}
Lastly, let $\hypgeo{2}{1}(a,b,c,z)$ be the Gaussian hypergeometric function with parameters $a$, $b$ and $c$ evaluated at $z$. The function is defined by the following series expansion
\begin{equation}\label{Hypergeometric_series}
    \hypgeo{2}{1}(a,b,c,z)=\frac{\Gamma(c)}{\Gamma(a)\Gamma(b)}\sum_{n=0}^\infty\frac{\Gamma(a+n)\Gamma(b+n)}{\Gamma(c+n)}\frac{z^n}{n!}
\end{equation}
with convergence guaranteed within the unit circle $|z|=1$. The first derivative of $\hypgeo{2}{1}(a,b,c,z)$ is given by
\begin{equation}\label{Hypergeometric_derivative}
    \hypgeo{2}{1}'(a,b,c,z)=\frac{ab}{c}\;\hypgeo{2}{1}(1+a,1+b,1+c,z)
\end{equation}
Note that the parameters of the special functions above can take complex values. Details on the properties of these special functions can be found in \cite{Abramowitz,Gradshteyn}

%%%%%%%%%%%%%%%%%%%%%%%%%%%%%%%%%%%%%%%%%%%%%%%%%%%%%%%%%%%%%%%%%%%%%%%%%%%%%%%%%%%%%%%%%%%%%%%

\section{Results}

\subsection{Analysis of expectations}

Let $E_A(t)$ be the expected number of cells of type $A$ at time $t$. Similarly, let $E_B(t)$ be the expected number of cells of type $B$ at time $t$. From (\ref{dynamics}), we have that $E_A$ satisfies
\begin{equation*}
    \frac{dE_A}{dt}=\lambda_A(1-2\alpha)E_A
\end{equation*}
Under the initial condition $E_A(0)=1$, the solution is $E_A(t)=e^{\lambda_A(1-2\alpha)t}$. Note that $E_A(t)$ is identically $1$ when $\alpha=1/2$, which is expected from a critical process. From the dynamics of cells of type $B$ in (\ref{dynamics}), we have that $E_B(t)$ must satisfy the following first-order linear differential equation
\begin{equation*}
    \frac{dE_B}{dt}=2\lambda_A(1-\alpha)e^{\lambda_A(1-2\alpha)t}+\lambda_B(q^2-p^2)E_B
\end{equation*}
At time $t=0$, we have no cells of type $B$. Therefore, the initial condition is simply $E_B(0)=0$. If all cells are critical (i.e., $\alpha=1/2$ and $p=q=1/2$; we call this case ``bi-critical''), the equation above simplifies, and the general solution is $E_B(t)=\lambda_At$. If the system is not bi-critical, we obtain the general solution
\begin{equation*}
    E_B(t)=\frac{2\lambda_A(1-\alpha)}{\lambda_A(1-2\alpha)-\lambda_B(q^2-p^2)}e^{\lambda_A(1-2\alpha)t}+Ce^{\lambda_B(q^2-p^2)t}
\end{equation*}
Under the initial condition $E_B(0)=0$, the particular solution is given by
\begin{equation*}
    E_B(t)=\frac{2\lambda_A(1-\alpha)}{\lambda_A(1-2\alpha)-\lambda_B(q^2-p^2)}\Big(e^{\lambda_A(1-2\alpha)t}-e^{\lambda_B(q^2-p^2)t}\Big)
\end{equation*}

%%%%%%%%%%%%%%%%%%%%%%%%%%%%%%%%%%%%%%%%%%%%%%%%%%%%%%%%%%%%%%%%%%%%%%%%%%%%%%%%%%%%%%%%%%%%%%%%%%%%

\subsection{Explicit probability generating function solution}

To derive the explicit solution, we consider the following backward equations \cite{KimmelAxelrod}:
\begin{equation}\label{eq:backward_A}
    \frac{dF_A}{dt}=-\lambda_AF_A+\lambda_Ah_A(F_A,F_B)
\end{equation}
\begin{equation}\label{eq:backward_B}
    \frac{dF_B}{dt}=-\lambda_BF_B+\lambda_Bh_B(F_A,F_B)
\end{equation}
To solve this system of backward equations, we will solve Equation (\ref{eq:backward_B}) for $F_B$ first and then substitute it into Equation (\ref{eq:backward_A}) to obtain $F_A$. The approach that we employ is largely inspired by \cite{Antal2010,Antal2011}, which also concerned closed-form expressions for the PGFs of two-type time continuous branching processes with exponentially distributed  lifetimes, but with different dynamics from ours. A two-stage model of carcinogenesis, which involved three cell types, was addressed in \cite{Denes} by reducing the partial differential equation satisfied by the PGF to the hypergeometric differential equation of Gauss.

%%%%%%%%%%%%%%%%%%%%%%%%%%%%%%%%%%%%%%%%%%%%%%%%%%%%%%%

\begin{theorem}
Given cells of type $A$ and $B$ that proliferate according to dynamics (\ref{dynamics}), the joint PGF $F_A(x,y,t)$ under different criticalities is given as follows.

\textbf{(a)} (\textbf{Bi-critical}) Let $\mu=\frac{4\lambda_A}{\lambda_B}$.
\begin{align}\label{FA_theorem1}
    F_A(w)=\frac{-4}{\mu}\bigg[&\frac{\mu}{2\sqrt{w}}\frac{\big(\frac{-\theta}{\mu\sqrt{w}}I_0^-(w)+I_1^-(w)\big)+C\big(\frac{\theta}{\mu\sqrt{w}}I_0^+(w)+I_1^+(w)\big)}{I_0^-(w)+CI_0^+(w)}\nonumber\\
    &+\frac{1}{2w}-\frac{\mu(1+w)}{4w}\bigg]
\end{align}
where $w=\frac{\lambda_Bt(1-y)/4+1}{1-y}$, $\theta=\sqrt{1-\mu}$ is a constant defined in terms of the model parameters, and $C=C(x)$ is a constant that is determined by the initial condition. Furthermore,
\begin{align}\label{I_theorem1}
    &I_0^-(w)=I(-\theta,\mu\sqrt{w});\ I_0^+(w)=I(\theta,\mu\sqrt{w});\nonumber\\
    &I_1^-(w)=I(-\theta+1,\mu\sqrt{w});\ I_1^+(w)=I(\theta+1,\mu\sqrt{w})
\end{align}

%%%%%%%%%%%%%%%%%%%%%%%%%%%%%%%%%%%%%%%%%%%%%%%%%%%%%%%

\textbf{(b)} (\textbf{Non-critical A-cells and critical B-cells}) Let $\mu=\frac{4\lambda_A}{\lambda_B}$. 
\begin{align}\label{FA_theorem2}
    F_A(w)=\frac{-1}{\mu(1-\alpha)^2}\bigg[&\frac{\theta}{2}-\frac{\theta_1}{w}+\frac{1}{w}\frac{\big(\frac{1}{2}+\theta_1+\theta_2\big)M_1(w)-CW_1(w)}{M_0(w)+CW_0(w)}\nonumber\\
    &-\frac{\mu(w+2\alpha(1-\alpha)(1-w))}{2w}\bigg]
\end{align}
where $w=\frac{\lambda_Bt(1-y)/4+1}{1-y}$, and $\theta_1$, $\theta_2$ and $\theta$ are constants defined in terms of the model parameters
\begin{equation*}
    \theta_1=-\frac{\mu(1-\alpha)\alpha}{|1-2\alpha|};\ \theta_2=\frac{1}{2}\sqrt{1-4\mu\alpha+4\mu\alpha^2};\ \theta=|1-2\alpha|\mu
\end{equation*}
$C$ is a constant that is determined by the initial condition, and
\begin{align}\label{MW_theorem2}
    M_0(w)=M(\theta_1,\theta_2,\theta w);\quad M_1(w)=M(\theta_1+1,\theta_2,\theta w);\nonumber\\
    W_0(w)=W(\theta_1,\theta_2,\theta w);\quad W_1(w)=W(\theta_1+1,\theta_2,\theta w)
\end{align}

%%%%%%%%%%%%%%%%%%%%%%%%%%%%%%%%%%%%%%%%%%%%%%%%%%%%%%%

\textbf{(c)} (\textbf{Non-critical B-cells}) Let $\mu=\frac{\lambda_A}{\lambda_B(p-q)}$.
\begin{align}\label{FA_theorem4}
    F_A(z)=\frac{-z}{\mu(1-\alpha)^2}\bigg\{&\bigg[\frac{1+\theta_1}{2}z^{-\frac{1-\theta_1}{2}}F_0^+(z)+\frac{(1+\theta_1+\theta_2)^2-\theta_3^2}{4q^2(1+\theta_1)}z^{\frac{1+\theta_1}{2}}F_1^+(z)\nonumber\\
    &+C\bigg(\frac{1-\theta_1}{2}z^{-\frac{1+\theta_1}{2}}F_0^-(z)+\frac{(1-\theta_1+\theta_2)^2-\theta_3^2}{4q^2(1-\theta_1)}z^{\frac{1-\theta_1}{2}}F_1^-(z)\bigg)\bigg]\nonumber\\
    &\bigg/\bigg[z^{\frac{1+\theta_1}{2}}F_0^+(z)+Cz^{\frac{1-\theta_1}{2}}F_0^-(z)\bigg]+\frac{1}{2}\frac{1+\theta_2}{z-q^2}\nonumber\\
    &-\frac{1}{2}\bigg[\frac{1}{z}-\frac{\mu}{z}\bigg(2\alpha(1-\alpha)\frac{z-p^2}{z-q^2}-1\bigg)\bigg]\bigg\}
\end{align}
where $z=\frac{p^2-yq^2}{1-y}e^{\lambda_B(p-q)t}$, and $\theta_1$, $\theta_2$ and $\theta_3$  are constants in terms of the model parameters
\begin{equation*}
    \theta_1=\frac{\mu\sqrt{q^2-4\alpha(1-\alpha)p^2}}{q};\ \theta_2=\frac{\sqrt{q^2-4\alpha(1-\alpha)\mu(p-q)}}{q};\ \theta_3=\mu|1-2\alpha|
\end{equation*}
$C$ is a constant that is determined by the initial condition, and
\begin{align}\label{F01_theorem4}
    &F_0^+(z)=\hypgeo{2}{1}\bigg(\frac{1}{2}(1+\theta_1+\theta_2+\theta_3),\frac{1}{2}(1+\theta_1+\theta_2-\theta_3),1+\theta_1,\frac{z}{q^2}\bigg);\nonumber\\
    &F_0^-(z)=\hypgeo{2}{1}\bigg(\frac{1}{2}(1-\theta_1+\theta_2+\theta_3),\frac{1}{2}(1-\theta_1+\theta_2-\theta_3),1-\theta_1,\frac{z}{q^2}\bigg);\nonumber\\
    &F_1^+(z)=\hypgeo{2}{1}\bigg(1+\frac{1}{2}(1+\theta_1+\theta_2+\theta_3),1+\frac{1}{2}(1+\theta_1+\theta_2-\theta_3),2+\theta_1,\frac{z}{q^2}\bigg);\nonumber\\
    &F_1^-(z)=\hypgeo{2}{1}\bigg(1+\frac{1}{2}(1-\theta_1+\theta_2+\theta_3),1+\frac{1}{2}(1-\theta_1+\theta_2-\theta_3),2-\theta_1,\frac{z}{q^2}\bigg)
\end{align}
\end{theorem}

%%%%%%%%%%%%%%%%%%%%%%%%%%%%%%%%%%%%%%%%%%%%%%%%%%%%%%%

\begin{proof}
\textbf{(a)} Under bi-criticality, we obtain the following backward equation for $B$-cells from (\ref{eq:progeny_B}) and (\ref{eq:backward_B})
\begin{equation*}
    \frac{dF_B}{dt}=\frac{\lambda_B}{4}(F_B-1)^2\nonumber
\end{equation*}
Noting that the initial condition is $F_B(x,y,0)=y$, we have
\begin{equation}\label{B_theorem1}
    F_B(x,y,t)=\frac{\lambda_Bt(1-y)/2+2y}{\lambda_Bt(1-y)/2+2}
\end{equation}
Combining (\ref{eq:progeny_A}), (\ref{eq:backward_A}) and (\ref{B_theorem1}) gives us
\begin{equation*}
    \frac{dF_A}{dt}=-\lambda_AF_A+\lambda_A\bigg((1-\alpha)F_A+\alpha\frac{\lambda_Bt(1-y)/2+2y}{\lambda_Bt(1-y)/2+2}\bigg)^2\nonumber
\end{equation*}
We then employ the transformation $f=F_B=\frac{\lambda_Bt(1-y)/2+2y}{\lambda_Bt(1-y)/2+2}$, which leads to
\begin{equation*}
    \frac{dF_A}{df}=\mu\frac{(1-\alpha)^2}{(1-f)^2}F_A^2+\mu\frac{2\alpha(1-\alpha)f-1}{(1-f)^2}F_A+\mu\alpha^2\frac{f^2}{(1-f)^2}
\end{equation*}
where $\mu=\frac{4\lambda_A}{\lambda_B}$ is a constant. We employ another transformation, $w=\frac{1}{1-f}$, to simplify the equation. Note that $\frac{dw}{df}=\frac{1}{(1-f)^2}=w^2$. Therefore, the equation becomes
\begin{equation*}
    \frac{dF_A}{dw}=\mu(1-\alpha)^2F_A^2+\mu\bigg(2\alpha(1-\alpha)\frac{w-1}{w}-1\bigg)F_A+\mu\alpha^2\bigg(\frac{w-1}{w}\bigg)^2
\end{equation*}
We have that $w(t=0)=\frac{1}{1-y}$. Hence, the initial condition becomes $F_A\Big(w=\frac{1}{1-y}\Big)=x$. This is a non-linear Riccati differential equation, which is not possible to solve in general. However, in this case we can convert it to a second-order linear ODE (the Sturm-Liouville equation) and try to solve it instead \cite{Bender}. To do this, we introduce the following short-hand notations
\begin{equation*}
    A=\mu(1-\alpha)^2;\ B=\mu\bigg(2\alpha(1-\alpha)\frac{w-1}{w}-1\bigg);\ C=\mu\alpha^2\bigg(\frac{w-1}{w}\bigg)^2
\end{equation*}
Following \cite{Antal2011}, we perform the transformation $F_A=\frac{-u'}{Au}=-\frac{1}{A}(\log u)'$. The equation becomes
\begin{equation*}
    u''+\gamma u'+\beta u=0
\end{equation*}
where the coefficients $\gamma$ and $\beta$ are given by
\begin{equation*}
    \gamma=-\bigg(\frac{A'}{A}+B\bigg)=\mu\bigg(1+2\alpha(1-\alpha)\frac{1-w}{w}\bigg);\ \beta=AC=\mu^2\alpha^2(1-\alpha)^2\bigg(\frac{1-w}{w}\bigg)^2
\end{equation*}
To remove the first derivative, we use another quasi-linear transformation $u=\Phi T$ such that $\Phi'=-\gamma\Phi/2$ \cite{Antal2011}. As we shall see, we do not need to compute the function $\Phi$ explicitly. Finally, we arrive at an equation whose solution can be obtained explicitly
\begin{equation}\label{T_second_ODE}
    T''-\bigg(\frac{\mu^2(w+4\alpha(1-\alpha)(1-w))}{4w}-\frac{\mu\alpha(1-\alpha)}{w^2}\bigg)T=0
\end{equation}
Since we are assuming that cells of type $A$ behave critically, the equation above reduces to
\begin{equation*}
    T''-\frac{\mu}{4w}\Big(\mu-\frac{1}{w}\Big)T=0
\end{equation*}
The closed-form solution for the ODE above can be written as a linear combination of modified-Bessel functions of the first kind. The solution, up to a multiplicative constant, is given by
\begin{equation*}
    T(w)=\sqrt{w}[I(-\sqrt{1-\mu},\mu\sqrt{w})+CI(\sqrt{1-\mu},\mu\sqrt{w})]
\end{equation*}
where $I(.)$ is the modified Bessel function defined in Section 2. Using the short-hand notations (\ref{I_theorem1}), we can write $T(w)=\sqrt{w}[I_0^-(w)+CI_0^+(w)]$. To obtain an explicit expression for $F_A$, we reverse the calculations as follows
\begin{equation}\label{reverse}
    F_A(w)=\frac{-1}{A}\bigg(\frac{T'(w)}{T(w)}-\frac{\gamma}{2}\bigg)
\end{equation}
As we can see, we are interested in the ratio between $T'(w)$ and $T(w)$ and, therefore, the omitted multiplicative constant does not matter. Using formulas (\ref{Bessel_derivative}) and (\ref{Bessel_recursive}), and short-hand notations (\ref{I_theorem1}), we arrive at the following expression for $T'(w)$
\begin{equation*}
    T'(w)=\frac{\mu}{2}\bigg[\bigg(\frac{-\theta}{\mu\sqrt{w}}I_0^-(w)+I_1^-(w)\bigg)+C\bigg(\frac{\theta}{\mu\sqrt{w}}I_0^+(w)+I_1^+(w)\bigg)\bigg]+\frac{1}{2w}T(w)
\end{equation*}
where $\theta=\sqrt{1-\mu}$. From (\ref{reverse}), given that $A=\frac{\mu}{4}$ and $\gamma=\frac{\mu(1+w)}{2w}$ in the critical case, we find that $F_A(w)$ is given by (\ref{FA_theorem1}), where $w=\frac{1}{1-f}=\frac{\lambda_Bt(1-y)/4+1}{1-y}$ is a function of $t$. The initial condition $F_A\big(w=\frac{1}{1-y}\big)=x$ leads to the following expression for the constant $C$.
\begin{equation}\label{C_theorem1}
    C=\frac{\kappa I_0^-\big(\frac{1}{1-y}\big)-\big(\frac{-\theta\sqrt{1-y}}{\mu}I_0^-\big(\frac{1}{1-y}\big)+I_1^-\big(\frac{1}{1-y}\big)\big)}{\big(\frac{\theta\sqrt{1-y}}{\mu}I_0^+\big(\frac{1}{1-y}\big)+I_1^+\big(\frac{1}{1-y}\big)\big)-\kappa I_0^+\big(\frac{1}{1-y}\big)}
\end{equation}
where
\begin{equation}\label{kappa_theorem1}
    \kappa=\frac{2-y-x}{2\sqrt{1-y}}-\frac{\sqrt{1-y}}{\mu}
\end{equation}
\end{proof}

\noindent Proofs of the remaining parts follow the same approach, and are given in details in the Supplement.

%%%%%%%%%%%%%%%%%%%%%%%%%%%%%%%%%%%%%%%%%%%%%%%%%%%%%%%%%%%%%%%%%%%%%%%%%%%%%%%%%%%%%%%%%%%%%%%%%%%%

\subsection{Asymptotic analysis}

In this section, we will study the asymptotic properties of the systems we consider, with particular focus on the probability of extinction of cells in large time.

%%%%%%%%%%%%%%%%%%%%%%%%%%%%%%%%%%%%%%%%%%%%%%%%%%%%%%%

\begin{theorem}
Suppose that cells of type $A$ and $B$ proliferate according to dynamics (\ref{dynamics}).

\textbf{(a)} (\textbf{Bi-critical}) Eventual extinction occurs with probability $1$. Furthermore, the rate at which the population approaches extinction is proportional to the square root of time.

\textbf{(b)} (\textbf{Non-critical A cells and critical B cells}) Eventual extinction happens with probability $\frac{\alpha^2}{(1-\alpha)^2}$ in the super-critical case, and with probability $1$ in the sub-critical case. In both cases, the rate at which the probability of extinction approaches the limit is proportional to time.

\textbf{(c)} (\textbf{Super-critical B cells}) Assuming $\theta_1\neq-1$, where $\theta_1$ is defined in part (c) of Theorem 1, eventual extinction happens with probability $\frac{1}{2(1-\alpha)^2}\Big[1-\frac{2\alpha(1-\alpha)p^2}{q^2}-\frac{\sqrt{q^2-4\alpha(1-\alpha)p^2}}{q}\Big]$. The probability of extinction approaches this limit exponentially with respect to time.

\end{theorem}

%%%%%%%%%%%%%%%%%%%%%%%%%%%%%%%%%%%%%%%%%%%%%%%%%%%%%%%

\begin{proof}
\textbf{(a)} The probability of extinction of the entire population by time $t$, $E(t)$, is given by $F_A(0,0,t)$. From (\ref{kappa_theorem1}), we have that, at $x=0$ and $y=0$,
\begin{equation*}
    \kappa_0\coloneqq\kappa|_{x=0,y=0}=1-\frac{1}{\mu}
\end{equation*}
Therefore, from (\ref{C_theorem1}),
\begin{equation*}
    C_0\coloneqq C|_{x=0,y=0}=\frac{\kappa_0I_0^-(1)-\big(\frac{-\theta}{\mu}I_0^-(1)+I_1^-(1)\big)}{\big(\frac{\theta}{\mu}I_0^+(1)+I_1^+(1)\big)-\kappa_0I_0^+(1)}
\end{equation*}
is a finite constant. Given that $w=\frac{\lambda_Bt(1-y)/4+1}{1-y}$, we have that $w_0\coloneqq w|_{x=0,y=0}=\frac{\lambda_B}{4}t+1$. Hence, $w_0\to\infty$ as $t\to\infty$. Using the large-argument approximation (\ref{Bessel_asymptotic}) for modified Bessel functions and the fact that $C_0$ is finite, we deduce that
\begin{equation*}
    \frac{\big(\frac{-\theta}{\mu\sqrt{w_0}}I_0^-(w_0)+I_1^-(w_0)\big)+C_0\big(\frac{\theta}{\mu\sqrt{w_0}}I_0^+(w_0)+I_1^+(w_0)\big)}{I_0^-(w_0)+C_0I_0^+(w_0)}\sim 1-\frac{\theta}{\mu\sqrt{w_0}}\frac{1-C_0}{1+C_0}
\end{equation*}
for large time. Therefore, it follows that 
\begin{align*}
    E(t)&=\frac{-4}{\mu}\bigg[\frac{\mu}{2\sqrt{w_0}}\frac{\big(\frac{-\theta}{\mu\sqrt{w_0}}I_0^-(w_0)+I_1^-(w_0)\big)+C_0\big(\frac{\theta}{\mu\sqrt{w_0}}I_0^+(w_0)+I_1^+(w_0)\big)}{I_0^-(w_0)+CI_0^+(w_0)}\nonumber\\
        &\qquad\quad+\frac{1}{2w_0}-\frac{\mu(1+w_0)}{4w_0}\bigg]\nonumber\\
        &\sim\frac{-4}{\mu}\bigg[\frac{\mu}{2\sqrt{w_0}}\bigg(1-\frac{\theta}{\mu\sqrt{w_0}}\frac{1-C_0}{1+C_0}\bigg)+\frac{1}{2w_0}-\frac{\mu(1+w_0)}{4w_0}\bigg]\nonumber\\
        &\sim\frac{-4}{\mu}\bigg(\frac{\mu}{2\sqrt{w_0}}-\frac{\mu}{4}\bigg)\nonumber\\
        &\sim 1-\frac{4/\sqrt{\lambda_B}}{\sqrt{t}}
\end{align*}
The theorem then follows. 
\end{proof}

\noindent Certain extinction of a two-type system with double criticalities has been proven in \cite{Athreya}. Hence, in the bi-critical case, our theorem agrees with established theoretical results in literature using a different approach. Detailed proofs of the remaining parts are given in the Supplement.

%%%%%%%%%%%%%%%%%%%%%%%%%%%%%%%%%%%%%%%%%%%%%%%%%%%%%%%%%%%%%%%%%%%%%%%%%%%%%%%%%%%%%%%%%%%%%%%%%%%%

\section{Discussion}

%%%%%%%%%%%%%%%%%%%%%%%%%%%%%%%%%%%%%%%%%%%%%%%%%%%%%%%%%%%%%%%%%%%%%%%%%%%%%%%%%%%%%%%%%%%%%%%%%%%%

\subsection{Stochastic models of hematopoiesis over time}

The history of stochastic models of hematopoiesis goes back to the work of Till, McCullogh, and Siminovitch in 1963 \cite{Till}, who modeled proliferation of colony-forming units (CFU) and described the distribution of sizes of resulting colonies using gamma distribution and simulated (using an early IBM machine) birth-death processes. Macken and Perelson \cite{Macken}, in a Springer Lecture Notes volume, proposed a series of more comprehensive models of hematopietic system in the form of a multi-type Galton-Watson process. This analysis was insightful, including probabilities of non-extinction (``completion of growth''), growth rates and the first two moments of cell counts. The results were mostly obtained using now classical theorems from Harris's book \cite{Harris}. Many original papers on stochastic models appeared since, including Dingli et al. \cite{Dingli} regarding stochastic clonal expansion of hematopoietic stem cells, Kimmel and Corey \cite{KimmelCorey} and Wojdyla et al. \cite{Wojdyla} regarding evolution of neutropenia-related leukemias, and many others. Some of them were reviewed in \cite{WhitchardBlood,kimmel2014stochasticity}. A paper that features a three-type branching process model not very unlike ours was published by Denes and Krewski \cite{Denes}. 

Diverse aspects of the dynamics of the hematopoietic systems can be understood using deterministic or quasi-stochastic approaches. The former involves expected-value dynamics modeled using ordinary differential equation (ODE) models, while the latter involves expected frequencies modeled using integral equations or partial differential equations of transport type. A comprehensive review of both types of approaches was carried out by Pujo-Menjouet \cite{Pujo}. A recent paper by Dinh et al. \cite{dinh2021predicting} shows how stochasticity is important in prediction of outcomes of treatments for leukemias (for the deterministic model related to the latter paper, see \cite{stiehl2014clonal}).

%%%%%%%%%%%%%%%%%%%%%%%%%%%%%%%%%%%%%%%%%%%%%%%%%%%%%%%%%%%%%%%%%%%%%%%%%%%%%%%%%%%%%%%%%%%%%%%%%%%%

\subsection{Review of biologically relevant cases}

\noindent We will limit the review to the haemopoietic (blood cell production) system in the Human and in laboratory animals such as the Mouse. The $A$-cells here are the haemopoietic stem cells (HST), while the $B$-cells are the committed multipotent progenitors (CMP), which have the ability to feed into more specialized compartments, with probability $p$ for each progeny cell following division. CMP's have self-renewal capacity, with probability $q$ for each progeny cell following division, but they mainly serve the role of transitional cells \cite{WhitchardBlood}. Analogies with other hierarchical cell production systems, such as the neurogenic progenitors in the brain \cite{villalba2020regulation}, urothelial cells in the urinary bladder \cite{wang2017urothelial}, and others, are immediate. However, it should be noted that, except for the neurogenesis, no other cell production system seems to reach the sophistication level of higher mammals haematopoiesis. Please note that we do not consider transients caused by regulatory feedbacks, such as the interferon $\gamma$ feedback in haematopoiesis \cite{hormaechea2021chronic}, which appears when the system is under stress.

\textbf{Normal adult haematopoiesis}: In this case, stabilization of the mean HSC and CMP counts is essential \cite{shepherd2004estimating}. This is only achieved if the $A$-cells are critical ($\alpha=1/2$) and the $B$-cells are sub-critical ($p>q$).  

\textbf{Normal fetal haematopoiesis}: In this case, growth of the mean HSC and CMP counts is essential \cite{Wojdyla}. This type of haematopoiesis occurs in the Human over the last 3 months of fetal development. This may be achieved if the $A$-cells are super-critical ($\alpha>1/2$) and the $B$-cells are sub-critical, critical or super-critical. In each of the 3 cases, the relative proportion of HSC to CMP will be changing differently with time. The probability of extinction in any finite time depends on the type of proliferation of $B$-cells. Inroads into understanding this process have been obtained using mosaic mice methodology \cite{ganuza2017lifelong}.

\textbf{Normal aging haematopoiesis}: In this case, the mean HSC count is decreasing but the CMP count should be prevented from falling too fast \cite{shepherd2004estimating}. This type of haematopoiesis occurs in the Human at age greater than about 70. This may be achieved if the $A$-cells are sub-critical ($\alpha>1/2$) and the $B$-cells are super-critical ($p<q$). 

\textbf{Repopulation following bone marrow transplant}: This case is not very dissimilar from the fetal haematopoiesis, as the main purpose is growth of the mean HSC and CMP counts. An interesting twist is that the transplant has to include a mix of some HSC and much more CMP, since transplanting HSC alone will not allow building up of haematopoiesis in a short time; growth of the mean HSC and CMP counts is essential \cite{stiehl2014impact}. Our model allows ballpark computations of this mix, given the minimum count of CMPs to be present after a fixed time from the transplant to prevent subject's death. 

\textbf{Cyclic neutropenias}: Mathematically interesting cases coincide frequently with biological reality. For example, in the bi-critical case, we expect large oscillations, observed in diseases such as cyclic neutropenias \cite{olofsen2020modeling}. In the historic paper by Mackey and Glass \cite{mackey1977oscillation}, cyclic hematopoiesis is attributed to defects in feedback governing stem cell activation, but the strong stochastic component may be better explained by criticality.

%%%%%%%%%%%%%%%%%%%%%%%%%%%%%%%%%%%%%%%%%%%%%%%%%%%%%%%%%%%%%%%%%%%%%%%%%%%%%%%%%%%%%%%%%%%%%%%%%%%%
\subsection{Distributions of cell lifetimes}

Methods, which we used to obtain our theorems, are typical for a linear birth-and-death process \cite{grimmett2020probability}. As it is known, the latter is mathematically identical to a binary fission branching process with exponential lifetimes, which traditionally is considered not appropriate for modeling cell populations. Indeed, in laboratory cell populations, the sojourn times in different cell cycle phases are usually of finite duration, and in some bacteria, cell divisions are almost synchronous, as attested by a voluminous literature of the subject (e.g., Kimmel and Axelrod's monograph \cite{KimmelAxelrod}). However, the situation is different in cancer and physiological control systems. For example, in lung cancer, time to doubling of cell populations is of the order of several weeks \cite{gorlova2005estimating}. In haematopoietic stem cells, it may be even longer \cite{catlin2011replication}. Much of this time is waiting for ``permission'' to divide (in healthy cells) or finishing of multiple rounds of DNA repair (in cancer cells). This part of the cell cycle is likely to be distributed exponentially, although measurements in vivo are difficult. In any case, models of physiological stem cell systems based on exponentially distributed cell lifetimes (hence expected values are described by ordinary differential equations) seem to work with considerable precision. This was demonstrated among others by measurements and mathematical models based on shortening of telomeres in haematopoietic stem cells \cite{sidorov2009leukocyte}.

%%%%%%%%%%%%%%%%%%%%%%%%%%%%%%%%%%%%%%%%%%%%%%%%%%%%%%%%%%%%%%%%%%%%%%%%%%%%%%%%%%%%%%%%%%%%%%%%%%%%

\section{Conclusions}

Two- and multi-type branching processes are minimalist models for normal cell production systems and cancer progression in which altered clones play a major role. We presented closed-form solutions for biologically meaningful two-type time continuous branching processes with exponentially distributed  lifetimes. We computed the probability generating functions explicitly for all possible combinations of cell-type-specific criticalities. The most practically interesting case, which arises in the dynamics of stem cells, is when the self-renewing stem cells are critical and the commited cells are sub-critical. However, as discussed previously, other system dynamics are also relevant in certain situations. While systems with non-critical dynamics are complex due to the presence of the confluent hypergeometric and Whittaker functions, solution for the bi-critical case can be expressed in terms of the relatively simple modified Bessel functions with well-defined asymptotic behaviors. Our asymptotic analyses not only reveal the large-time probability of extinction, but also describe the rate at which this probability approaches the limit. Given the closed-form solutions, other interesting asymptotic properties can be studied by considering certain scaling limits, and require further research in the future. It is also more realistic to model cancer progression using a multi-type branching process. Given the already complex nature of the explicit solutions presented in our paper, doing the same for a multi-type system may not be feasible (however, see \cite{Denes}) and require a different approach. 

%%%%%%%%%%%%%%%%%%%%%%%%%%%%%%%%%%%%%%%%%%%%%%%%%%%%%%%%%%%%%%%%%%%%%%%%%%%%%%%%%%%%%%%%%%%%%%%%%%%%

\small

\section*{Acknowledgments}

Nam Nguyen acknowledges funding from CPRIT grant RP200383 (Dr. Wenyi Wang, PI). Marek Kimmel acknowledges funding from NIH R01HL136333 and R01HL134880 grants (Dr. Katherine King, PI).

%%%%%%%%%%%%%%%%%%%%%%%%%%%%%%%%%%%%%%%%%%%%%%%%%%%%%%%%%%%%%%%%%%%%%%%%%%%%%%%%%%%%%%%%%%%%%%%%%%%%

\bibliographystyle{plain}
\bibliography{references}

\begin{thebibliography}{10}

\bibitem{Abramowitz}
Milton Abramowitz and Irene~A. Stegun.
\newblock {\em Handbook of Mathematical Functions with Formulas, Graphs, and
  Mathematical Tables}.
\newblock Dover, New York, ninth edition, 1964.

\bibitem{Antal2010}
Tibor Antal and Pavel Krapivsky.
\newblock Exact solution of a two-type branching process: Clone size
  distribution in cell division kinetics.
\newblock {\em J Stat Mech}, 2009.

\bibitem{Antal2011}
Tibor Antal and Pavel Krapivsky.
\newblock Exact solution of a two-type branching process: Models of tumor
  progression.
\newblock {\em J Statist Mech Theory Exp}, 8, 2011.

\bibitem{Athreya}
K.B. Athreya, P.E. Ney, and P.E. Ney.
\newblock {\em Branching Processes}.
\newblock Dover Books on Mathematics. Dover Publications, 2004.

\bibitem{Bender}
Carl Bender and Steven Orszag.
\newblock {\em Advanced Mathematical Methods for Scientists and Engineers:
  Asymptotic Methods and Perturbation Theory}, volume~1.
\newblock Springer, New York, NY, 1999.

\bibitem{catlin2011replication}
Sandra~N Catlin, Lambert Busque, Rosemary~E Gale, Peter Guttorp, and Janis~L
  Abkowitz.
\newblock The replication rate of human hematopoietic stem cells in vivo.
\newblock {\em Blood, The Journal of the American Society of Hematology},
  117(17):4460--4466, 2011.

\bibitem{Denes}
Josef Denes and Daniel Krewski.
\newblock An exact representation for the generating function for the
  moolgavkar-venzon-knudson two-stage model of carcinogenesis with stochastic
  stem cell growth.
\newblock {\em Mathematical Biosciences}, 131(2):185--204, 1996.

\bibitem{Dingli}
David Dingli, Arne Traulsen, and Jorge~M. Pacheco.
\newblock Stochastic dynamics of hematopoietic tumor stem cells.
\newblock {\em Cell Cycle}, 6:461--466, 2007.

\bibitem{dinh2021predicting}
Khanh~N Dinh, Roman Jaksik, Seth~J Corey, and Marek Kimmel.
\newblock Predicting time to relapse in acute myeloid leukemia through
  stochastic modeling of minimal residual disease based on clonality data.
\newblock {\em Computational and Systems Oncology}, 1(3):e1026, 2021.

\bibitem{ganuza2017lifelong}
Miguel Ganuza, Trent Hall, David Finkelstein, Ashley Chabot, Guolian Kang, and
  Shannon McKinney-Freeman.
\newblock Lifelong haematopoiesis is established by hundreds of precursors
  throughout mammalian ontogeny.
\newblock {\em Nature cell biology}, 19(10):1153--1163, 2017.

\bibitem{gorlova2005estimating}
Olga Gorlova, Bo~Peng, David Yankelevitz, Claudia Henschke, and Marek Kimmel.
\newblock Estimating the growth rates of primary lung tumours from samples with
  missing measurements.
\newblock {\em Statistics in medicine}, 24(7):1117--1134, 2005.

\bibitem{Gradshteyn}
I.~S. Gradshteyn and I.~M. Ryzhik.
\newblock {\em Table of integrals, series, and products}.
\newblock Elsevier/Academic Press, Amsterdam, seventh edition, 2007.
\newblock Translated from the Russian, Translation edited and with a preface by
  Alan Jeffrey and Daniel Zwillinger, With one CD-ROM (Windows, Macintosh and
  UNIX).

\bibitem{grimmett2020probability}
G.~Grimmett and D.~Stirzaker.
\newblock {\em Probability and random processes}.
\newblock Oxford university press, 2020.

\bibitem{Harris}
T.~E. Harris.
\newblock {Branching Processes}.
\newblock {\em The Annals of Mathematical Statistics}, 19(4):474--494, 1948.

\bibitem{hormaechea2021chronic}
Daniel Hormaechea-Agulla, Katie~A Matatall, Duy~T Le, Bailee Kain, Xiaochen
  Long, Pawel Kus, Roman Jaksik, Grant~A Challen, Marek Kimmel, and Katherine~Y
  King.
\newblock Chronic infection drives dnmt3a-loss-of-function clonal hematopoiesis
  via ifn$\gamma$ signaling.
\newblock {\em Cell Stem Cell}, 2021.

\bibitem{kimmel2014stochasticity}
Marek Kimmel.
\newblock Stochasticity and determinism in models of hematopoiesis.
\newblock {\em In: A systems biology approach to blood; Corey, Kimmel and
  Leonard, eds.}, pages 119--152, 2014.

\bibitem{KimmelAxelrod}
Marek Kimmel and David Axelrod.
\newblock {\em Branching Processes in Biology}, volume~19.
\newblock Springer, New York, NY, 2002.

\bibitem{KimmelCorey}
Marek Kimmel and Seth Corey.
\newblock Stochastic hypothesis of transition from inborn neutropenia to aml:
  Interactions of cell population dynamics and population genetics.
\newblock {\em Frontiers in Oncology}, 3:89, 2013.

\bibitem{Macken}
C.A. Macken and A.S. Perelson.
\newblock {\em Stem Cell Proliferation and Differentiation: A Multitype
  Branching Process Model}.
\newblock Lecture Notes in Biomathematics. Springer Berlin Heidelberg, 1988.

\bibitem{mackey1977oscillation}
Michael~C Mackey and Leon Glass.
\newblock Oscillation and chaos in physiological control systems.
\newblock {\em Science}, 197(4300):287--289, 1977.

\bibitem{olofsen2020modeling}
Patricia~A Olofsen and Ivo~P Touw.
\newblock Modeling severe congenital neutropenia in induced pluripotent stem
  cells.
\newblock In {\em Recent Advances in iPSC Disease Modeling, Volume 1}, pages
  85--101. Elsevier, 2020.

\bibitem{Pujo}
Laurent Pujo-Menjouet.
\newblock Blood cell dynamics: Half of a century of modelling.
\newblock {\em Mathematical Modelling of Natural Phenomena}, 11:92--115, 2016.

\bibitem{shepherd2004estimating}
Bryan~E Shepherd, Peter Guttorp, Peter~M Lansdorp, and Janis~L Abkowitz.
\newblock Estimating human hematopoietic stem cell kinetics using granulocyte
  telomere lengths.
\newblock {\em Experimental hematology}, 32(11):1040--1050, 2004.

\bibitem{sidorov2009leukocyte}
Igor Sidorov, Masayuki Kimura, Anatoli Yashin, and Abraham Aviv.
\newblock Leukocyte telomere dynamics and human hematopoietic stem cell
  kinetics during somatic growth.
\newblock {\em Experimental hematology}, 37(4):514--524, 2009.

\bibitem{stiehl2014impact}
T~Stiehl, AD~Ho, and A~Marciniak-Czochra.
\newblock The impact of cd34+ cell dose on engraftment after scts: personalized
  estimates based on mathematical modeling.
\newblock {\em Bone marrow transplantation}, 49(1):30--37, 2014.

\bibitem{stiehl2014clonal}
Thomas Stiehl, Natalia Baran, Anthony~D Ho, and Anna Marciniak-Czochra.
\newblock Clonal selection and therapy resistance in acute leukaemias:
  mathematical modelling explains different proliferation patterns at diagnosis
  and relapse.
\newblock {\em Journal of The Royal Society Interface}, 11(94):20140079, 2014.

\bibitem{Till}
James~E. Till, Ernest~A. McCulloch, and Louis Siminovitch.
\newblock A stochastic model of stem cell proliferation, based on the growth of
  spleen colony-forming cells.
\newblock {\em Proceedings of the National Academy of Sciences of the United
  States of America}, 51:29--36, 1964.

\bibitem{villalba2020regulation}
Ana Villalba, Magdalena G{\"o}tz, and V{\'\i}ctor Borrell.
\newblock The regulation of cortical neurogenesis.
\newblock {\em Current topics in developmental biology}, 142:1--66, 2020.

\bibitem{wang2017urothelial}
Caihong Wang, Whitney~Trotter Ross, and Indira~U Mysorekar.
\newblock Urothelial generation and regeneration in development, injury, and
  cancer.
\newblock {\em Developmental Dynamics}, 246(4):336--343, 2017.

\bibitem{WhitchardBlood}
Zakary~L. Whichard, Casim~A. Sarkar, Marek Kimmel, and Seth~J. Corey.
\newblock Hematopoiesis and its disorders: a systems biology approach.
\newblock {\em Blood}, 115(12):2339--2347, 2010.

\bibitem{Wojdyla}
Tomasz Wojdyla, Hrishikesh Mehta, Taly Glaubach, Roberto Bertolusso, Marta
  Iwanaszko, Rosemary Braun, Seth~J. Corey, and Marek Kimmel.
\newblock Mutation, drift and selection in single-driver hematologic
  malignancy: Example of secondary myelodysplastic syndrome following treatment
  of inherited neutropenia.
\newblock {\em PLOS Computational Biology}, 15(1):1--22, 2019.

\end{thebibliography}

%%%%%%%%%%%%%%%%%%%%%%%%%%%%%%%%%%%%%%%%%%%%%%%%%%%%%%%%%%%%%%%%%%%%%%%%%%%%%%%%%%%%%%%%%%%%%%%%%%%%

\newpage

\singlespacing

\begin{center}
    {\Large\bf Supplementary Materials for}\\
    \vspace{3mm}
    {\large``Stochastic Models of Stem Cells and Their Descendants under Different Criticality Assumptions''}
\end{center}
 
\begin{center}
    \large Nam H Nguyen and Marek Kimmel
\end{center}
\begin{center}
    \large Department of Statistics, Rice University, Houston, TX, USA
\end{center}

%%%%%%%%%%%%%%%%%%%%%%%%%%%%%%%%%%%%%%%%%%%%%%%%%%%%%%%%%%%%%%%%%%%%

\newpage

\noindent In this supplement, we will give detailed proofs that have been not covered in the main article. We will frequently refer to equations and expressions in the main text.

\subsection*{Proof of Theorem 1b}

Most of the calculations in the proof of Theorem 1a remain unchanged. We start from the second-order ODE (\ref{T_second_ODE}). Up to a multiplicative constant, the general solution to this second-order linear ODE can be expresed in terms of the Whittaker functions
\begin{equation*}
    T(w)=M(\theta_1,\theta_2,\theta w)+CW(\theta_1,\theta_2,\theta w)
\end{equation*}
where
\begin{equation*}
\theta_1=-\frac{\mu(1-\alpha)\alpha}{|1-2\alpha|};\ \theta_2=\frac{1}{2}\sqrt{1-4\mu\alpha+4\mu\alpha^2};\ \theta=|1-2\alpha|\mu
\end{equation*}
Using formulae for the derivative of a Whittaker function (\ref{Whittaker_derivative}), and short-hand notations (\ref{MW_theorem2}), we can show that
\begin{equation*}
    T'(w)=\Big(\frac{\theta}{2}-\frac{\theta_1}{w}\Big)(M_0(w)+CW_0(w))+\frac{1}{w}\Big[\Big(\frac{1}{2}+\theta_1+\theta_2\Big)M_1(w)-CW_1(w)\Big]
\end{equation*}
We reverse the transformations using (\ref{reverse}). Then, the explicit expression of $F_A(w)$ is given by (\ref{FA_theorem2}), where $w=\frac{1}{1-f}=\frac{\lambda_Bt(1-y)/4+1}{1-y}$ is a function of $t$. Finally, using the initial condition $F_A\Big(w=\frac{1}{1-y}\Big)=x$, we can deduce the following expression for the constant $C$.
\begin{equation*}
    C=\frac{(\frac{1}{2}+\theta_1+\theta_2\big)M_1\big(\frac{1}{1-y}\big)-\kappa M_0\big(\frac{1}{1-y}\big)}{\kappa W_0\big(\frac{1}{1-y}\big)+W_1\big(\frac{1}{1-y}\big)}
\end{equation*}
where
\begin{equation*}
    \kappa=\frac{1}{1-y}\Big(-\mu(1-\alpha)^2x+\frac{\mu(1-2\alpha(1-\alpha)y)}{2}-\frac{\theta}{2}+\theta_1(1-y)\Big)
\end{equation*}

%%%%%%%%%%%%%%%%%%%%%%%%%%%%%%%%%%%%%%%%%%%%%%%%%%%%%%%%%%%%%%%%%%%%

\subsection*{Proof of Theorem 1c}

Combining (\ref{eq:progeny_B}) and (\ref{eq:backward_B}), we obtain the following backward equation for $F_B$
\begin{equation*}
    \frac{dF_B}{dt}=\lambda_B(q^2F_B^2+2pqF_B-F_B+q^2)
\end{equation*}
Using basic integration, we get the following general solution
\begin{equation*}
    F_B(x,y,t)=\frac{Pe^{\lambda_B(p-q)t}-p^2}{Pe^{\lambda_B(p-q)t}-q^2}
\end{equation*}
Using the initial condition $F_B(x,y,0)=y$, we find that $P=\frac{p^2-yq^2}{1-y}$. From (\ref{eq:backward_A}), we need to solve the backward equation for $F_A$
\begin{equation*}
    \frac{dF_A}{dt}=-\lambda_A+\lambda_A((1-\alpha)F_A+\alpha F_B)^2
\end{equation*}
Write $z=Pe^{\lambda_B(p-q)t}$. Then, $F_B(z)=\frac{z-p^2}{z-q^2}$. Note also that $\frac{dz}{dt}=\lambda_B(p-q)z$. Under this transformation, we end up with the following equation for $F_A$ in terms of $z$
\begin{equation*}
    \frac{dF_A}{dz}\lambda_B(p-q)z=-\lambda_AF_A+\lambda_A\bigg((1-\alpha)F_A+\alpha\frac{z-p^2}{z-q^2}\bigg)^2
\end{equation*}
Denoting $\mu=\frac{\lambda_A}{\lambda_B(p-q)}$, 
\begin{equation*}
    \frac{dF_A}{dz}=\frac{\mu}{z}(1-\alpha)^2F_A^2+\frac{\mu}{z}\bigg(2\alpha(1-\alpha)\frac{z-p^2}{z-q^2}-1\bigg)F_A+\frac{\mu}{z}\alpha^2\bigg(\frac{z-p^2}{z-q^2}\bigg)^2
\end{equation*}
Let's denote the coefficients of $F_A^2$, $F_A$ and the constant term by $A$, $B$ and $C$ respectively
\begin{align*}
    A=\frac{\mu(1-\alpha)^2}{z};\ B=\frac{\mu}{z}\bigg(2\alpha(1-\alpha)\frac{z-p^2}{z-q^2}-1\bigg);\ C=\frac{\mu\alpha^2}{z}\bigg(\frac{z-p^2}{z-q^2}\bigg)^2
\end{align*}
As before, this non-linear ODE can be transformed to a Sturm-Liouville differential equation 
\begin{equation*}
    u''+\gamma u'+\beta u=0
\end{equation*}
where $\gamma$ and $\beta$ are given by
\begin{equation*}
    \gamma=-\bigg(\frac{A'}{A}+B\bigg)=\frac{1}{z}-\frac{\mu}{z}\bigg(2\alpha(1-\alpha)\frac{z-p^2}{z-q^2}-1\bigg);\ \beta=AC=\bigg(\frac{\mu\alpha(1-\alpha)}{z}\bigg)^2\bigg(\frac{z-p^2}{z-q^2}\bigg)^2
\end{equation*}
Then, we apply the transformation $u=\Phi T$, where $\Phi$ is a function of $z$ that satisfies the condition $\Phi'=-\gamma\Phi/2$. The equation above reduces to
\begin{equation*}
    T''+\bigg[\frac{1-\mu^2}{4z^2}+\frac{\mu\alpha(1-\alpha)}{z(z-q^2)}\bigg(\frac{\mu(z-p^2)}{z}+\frac{p-q}{z-q^2}\bigg)\bigg]T=0
\end{equation*}
The solution to this ODE can be written in terms of the hypergeometric functions
\begin{equation*}
    T(z)=(z-q^2)^{\frac{1+\theta_2}{2}}\bigg(z^{\frac{1+\theta_1}{2}}F_0^+(z)+Cz^{\frac{1-\theta_1}{2}}F_0^-(z)\bigg)
\end{equation*}
where
\begin{equation*}
    \theta_1=\frac{\mu\sqrt{q^2-4\alpha(1-\alpha)p^2}}{q};\ \theta_2=\frac{\sqrt{q^2+4\alpha(1-\alpha)\mu(q-p)}}{q};\ \theta_3=\mu|1-2\alpha|
\end{equation*}
and $F_0^+$ and $F_0^-$ are short-hand notations defined in (\ref{F01_theorem4}). Using formulae for the first derivative of a hypergeometric function (\ref{Hypergeometric_derivative}), we have that
\begin{align*}
    T'(z)=(z-q^2)^{\frac{1+\theta_2}{2}}&\bigg[\frac{1+\theta_1}{2}z^{-\frac{1-\theta_1}{2}}F_0^+(z)+\frac{(1+\theta_1+\theta_2)^2-\theta_3^2}{4q^2(1+\theta_1)}z^{\frac{1+\theta_1}{2}}F_1^+(z)\\
    &+C\bigg(\frac{1-\theta_1}{2}z^{-\frac{1+\theta_1}{2}}F_0^-(z)+\frac{(1-\theta_1+\theta_2)^2-\theta_3^2}{4q^2(1-\theta_1)}z^{\frac{1-\theta_1}{2}}F_1^-(z)\bigg)\bigg]\\
    &+\frac{1+\theta_2}{2(z-q^2)}T(z)
\end{align*}
By reversing the transformations using (\ref{reverse}), we find that the explicit expression of $F_A(z)$ is given by (\ref{FA_theorem4}). The initial condition in terms of $z$ is $F_A\Big(z=\frac{p^2-yq^2}{1-y}\Big)=x$. The constant $C$ is given as follows.
\begin{align*}
    C=&\bigg[\bar{z}^{\frac{1+\theta_1}{2}}F_0^+(\bar{z})\kappa-\bigg(\frac{1+\theta_1}{2}\bar{z}^{-\frac{1-\theta_1}{2}}F_0^+(\bar{z})+\frac{(1+\theta_1+\theta_2)^2-\theta_3^2}{4q^2(1+\theta_1)}\bar{z}^{\frac{1+\theta_1}{2}}F_1^+(\bar{z})\bigg)\bigg]\nonumber\\
    &\bigg/\bigg[\bigg(\frac{1-\theta_1}{2}\bar{z}^{-\frac{1+\theta_1}{2}}F_0^-(\bar{z})+\frac{(1-\theta_1+\theta_2)^2-\theta_3^2}{4q^2(1-\theta_1)}\bar{z}^{\frac{1-\theta_1}{2}}F_1^-(\bar{z})\bigg)-\bar{z}^{\frac{1-\theta_1}{2}}F_0^-(\bar{z})\kappa\bigg]
\end{align*}
where $\bar{z}=\frac{p^2-yq^2}{1-y}$, and
\begin{equation*}
    \kappa=\frac{1-y}{p^2-yq^2}\bigg(\frac{-\mu(x+y)}{4}+\frac{1+\mu}{2}-\frac{(1+\theta_2)(p^2-yq^2)}{2(p-q)}\bigg)
\end{equation*}

%%%%%%%%%%%%%%%%%%%%%%%%%%%%%%%%%%%%%%%%%%%%%%%%%%%%%%%%%%%%%%%%%%%%

\subsection*{Proof of Theorem 2b}

The probability of extinction by time $t$, $E(t)$, is $F_A(0,0,t)$, where $F_A$ is given by (\ref{FA_theorem2}). From Theorem 1b, at $x=0$ and $y=0$,
\begin{equation*}
    C_0\coloneqq C|_{x=0,y=0}=\frac{(\frac{1}{2}+\theta_1+\theta_2\big)M_1(1)-\kappa_0M_0(1)}{\kappa_0W_0(1)+W_1(1)}
\end{equation*}
where
\begin{equation*}
    \kappa_0\coloneqq\kappa|_{x=0,y=0}=\frac{\mu}{2}-\frac{\theta}{2}+\theta_1
\end{equation*}
and $\theta_1$, $\theta_2$ and $\theta$ are constants defined in terms of the model parameters, $\mu$ and $\alpha$, only. Therefore, it follows that $C_0$ is a finite constant.\\
To study the asymptotic properties of the system, we use the algebraic relationship (\ref{Whittaker_Hypergeometric}) between the Whittaker functions and the confluent hypergeometric functions. Note that, in Theorem 1b, $w_0\coloneqq w|_{x=0,y=0}=\frac{\lambda_B}{4}t+1$ at $x=0$ and $y=0$. Hence, $w_0\to\infty$ as $t\to\infty$, and we can use the large-argument properties (\ref{Hypergeometric_asymptotic}) of $\hypgeo{1}{1}(.)$ and $U(.)$. Due to the presence of the exponential term, it is clear that the $\hypgeo{1}{1}(.)$ dominates $U(.)$ in magnitude when the argument is large. Thus, $M(.)$ dominates $W(.)$ in absolute value asymptotically. This, coupled with the knowledge that $C_0$ is finite in large time, leads to
\begin{equation*}
    \frac{\big(\frac{1}{2}+\theta_1+\theta_2\big)M_1(w_0)-C_0W_1(w_0)}{M_0(w_0)+C_0W_0(w_0)}\sim\bigg(\frac{1}{2}+\theta_1+\theta_2\bigg)\frac{M_1(w_0)}{M_0(w_0)}
\end{equation*}
Again, using (\ref{Whittaker_Hypergeometric}), and the approximation of $\hypgeo{1}{1}(.)$ as given by (\ref{Hypergeometric_asymptotic}), 
\begin{equation*}
    \frac{M_1(w_0)}{M_0(w_0)}\sim\frac{\theta_2-\theta_1-\frac{1}{2}}{\theta w_0}
\end{equation*}
Then, it follows that
\begin{align*}
    E(t)&=\frac{-1}{\mu(1-\alpha)^2}\bigg(\frac{\theta}{2}-\frac{\theta_1}{w_0}+\frac{1}{w_0}\frac{\big(\frac{1}{2}+\theta_1+\theta_2\big)M_1(w_0)-C_0W_1(w_0)}{M_0(w_0)+C_0W_0(w_0)}\\
    &\qquad\qquad\qquad-\frac{\mu(w_0+2\alpha(1-\alpha)(1-w_0))}{2w_0}\bigg)\\
    &\sim\frac{-1}{\mu(1-\alpha)^2}\bigg(\frac{\theta}{2}-\frac{\theta_1}{w_0}+\frac{1}{w_0}\frac{\big(\frac{1}{2}+\theta_1+\theta_2\big)\big(\theta_2-\theta_1-\frac{1}{2}\big)}{\theta w_0}\\
    &\qquad\qquad\qquad-\frac{\mu(w_0+2\alpha(1-\alpha)(1-w_0))}{2w_0}\bigg)\\
    &\sim\frac{-1}{\mu(1-\alpha)^2}\bigg(\frac{\theta}{2}-\frac{\mu(1-2\alpha(1-\alpha))}{2}-\frac{\theta_1}{w_0}-\frac{\mu\alpha(1-\alpha)}{w_0}\bigg)\\
    &=\frac{-1}{\mu(1-\alpha)^2}\bigg(\frac{\mu|1-2\alpha|}{2}-\frac{\mu(1-2\alpha(1-\alpha))}{2}+\frac{\mu\alpha(1-\alpha)}{|1-2\alpha|w_0}-\frac{\mu\alpha(1-\alpha)}{w_0}\bigg)\\
    &=\begin{cases} \frac{\alpha^2}{(1-\alpha)^2}-\frac{2\alpha^2}{(1-\alpha)(1-2\alpha)w_0}, & \mbox{if } 0<\alpha<\frac{1}{2} \\ 1-\frac{2\alpha}{(2\alpha-1)w_0}, & \mbox{if } \frac{1}{2}<\alpha<1 \end{cases}\\
    &\sim\begin{cases} \frac{\alpha^2}{(1-\alpha)^2}-\frac{8\alpha^2}{(1-\alpha)(1-2\alpha)\lambda_Bt}, & \mbox{if } 0<\alpha<\frac{1}{2} \\ 1-\frac{8\alpha}{(2\alpha-1)\lambda_Bt}, & \mbox{if } \frac{1}{2}<\alpha<1 \end{cases}
\end{align*}
The theorem then follows. 

%%%%%%%%%%%%%%%%%%%%%%%%%%%%%%%%%%%%%%%%%%%%%%%%%%%%%%%%%%%%%%%%%%%%

\subsection*{Proof of Theorem 2c}

The probability of extinction by time $t$, $E(t)$, is $F_A(0,0,t)$, where $F_A$ is given by (\ref{FA_theorem4}). Recall from Theorem 1c that $\theta_1=\frac{\mu\sqrt{q^2-4\alpha(1-\alpha)p^2}}{q}$. In the super-critical case, we assume $p<q$ and, therefore, $\theta_1$ is real for all $\alpha\in[0,1]$. Similarly, we can easily show that $\theta_2$ and $\theta_3$ are real constants defined in terms of the model parameters for all $\alpha$. From Theorem 1c, at $x=0$ and $y=0$,
\begin{equation*}
    \bar{z}_0\coloneqq\bar{z}|_{x=0,y=0}=p^2
\end{equation*}
where
\begin{equation*}
    \kappa_0\coloneqq\kappa|_{x=0,y=0}=\frac{1}{p^2}\bigg(\frac{1+\mu}{2}-\frac{(1+\theta_2^2)p^2}{2(p-q)}\bigg)
\end{equation*}
are real and finite. It then follows that
\begin{align*}
    C_0&\coloneqq C|_{x=0,y=0}\\
    &=\bigg[\bar{z_0}^{\frac{1+\theta_1}{2}}F_0^+(\bar{z}_0)\kappa_0-\bigg(\frac{1+\theta_1}{2}\bar{z}_0^{-\frac{1-\theta_1}{2}}F_0^+(\bar{z}_0)+\frac{(1+\theta_1+\theta_2)^2-\theta_3^2}{4q^2(1+\theta_1)}\bar{z}_0^{\frac{1+\theta_1}{2}}F_1^+(\bar{z}_0)\bigg)\bigg]\\
    &\ \bigg/\bigg[\bigg(\frac{1-\theta_1}{2}\bar{z}_0^{-\frac{1+\theta_1}{2}}F_0^-(\bar{z}_0)+\frac{(1-\theta_1+\theta_2)^2-\theta_3^2}{4q^2(1-\theta_1)}\bar{z}_0^{\frac{1-\theta_1}{2}}F_1^-(\bar{z}_0)\bigg)-\bar{z}_0^{\frac{1-\theta_1}{2}}F_0^-(\bar{z}_0)\kappa_0\bigg]
\end{align*}
is finite. Recall that $z=\frac{p^2-yq^2}{1-y}e^{\lambda_B(p-q)t}$. At $x=0$ and $y=0$, $z_0\coloneqq z|_{x=0,y=0}=p^2e^{\lambda_B(p-q)t}$. Since $p<q$, we have that $z_0\to 0$ as $t\to\infty$. From the series expansion of the hypergeometric function (\ref{Hypergeometric_series}), it is easy to see that hypergeometric functions are identically $1$ when evaluated at $0$. Setting $w_0=\frac{1}{z_0}$, we have
\begin{align*}
    &\bigg[\frac{1+\theta_1}{2}z_0^{-\frac{1-\theta_1}{2}}F_0^+(z_0)+\frac{(1+\theta_1+\theta_2)^2-\theta_3^2}{4q^2(1+\theta_1)}z_0^{\frac{1+\theta_1}{2}}F_1^+(z_0)\\
    &+C_0\bigg(\frac{1-\theta_1}{2}z_0^{-\frac{1+\theta_1}{2}}F_0^-(z_0)+\frac{(1-\theta_1+\theta_2)^2-\theta_3^2}{4q^2(1-\theta_1)}z_0^{\frac{1-\theta_1}{2}}F_1^-(z_0)\bigg)\bigg]\\
    &\bigg/\bigg[z_0^{\frac{1+\theta_1}{2}}F_0^+(z_0)+C_0z_0^{\frac{1-\theta_1}{2}}F_0^-(z_0)\bigg]\\
    &\sim\bigg[\frac{1+\theta_1}{2}w_0^{\frac{1-\theta_1}{2}}+\frac{(1+\theta_1+\theta_2)^2-\theta_3^2}{4q^2(1+\theta_1)}w_0^{-\frac{1+\theta_1}{2}}\\
    &+C_0\bigg(\frac{1-\theta_1}{2}w_0^{\frac{1+\theta_1}{2}}+\frac{(1-\theta_1+\theta_2)^2-\theta_3^2}{4q^2(1-\theta_1)}w_0^{-\frac{1-\theta_1}{2}}\bigg)\bigg]\\
    &\bigg/\bigg[w_0^{-\frac{1+\theta_1}{2}}+C_0w_0^{-\frac{1-\theta_1}{2}}\bigg]
\end{align*}
Clearly, $w_0\to\infty$ as $t\to\infty$, so we are interested in the behavior of the expression above when $w_0$ is large. Since $\mu=\frac{\lambda_A}{\lambda_B(p-q)}$ is negative when $p<q$, it follows that $\theta_1$ is also negative. By comparing the exponents, $w_0^{\frac{1-\theta_1}{2}}$ is the dominant term in the numerator, and $w_0^{-\frac{1+\theta_1}{2}}$ is the dominant term in the denominator. This, coupled with the fact that $C_0$ is finite, gives us the following approximation
\begin{align*}
    &\bigg[\frac{1+\theta_1}{2}w_0^{\frac{1-\theta_1}{2}}+\frac{(1+\theta_1+\theta_2)^2-\theta_3^2}{4q^2(1+\theta_1)}w_0^{-\frac{1+\theta_1}{2}}\\
    &+C_0\bigg(\frac{1-\theta_1}{2}w_0^{\frac{1+\theta_1}{2}}+\frac{(1-\theta_1+\theta_2)^2-\theta_3^2}{4q^2(1-\theta_1)}w_0^{-\frac{1-\theta_1}{2}}\bigg)\bigg]\\
    &\bigg/\bigg[w_0^{-\frac{1+\theta_1}{2}}+C_0w_0^{-\frac{1-\theta_1}{2}}\bigg]\sim\frac{1+\theta_1}{2}w_0
\end{align*}

\noindent Assuming $-1<\theta_1<0$, it follows that
\begin{align*}
    &\bigg[\frac{1+\theta_1}{2}w_0^{\frac{1-\theta_1}{2}}+\frac{(1+\theta_1+\theta_2)^2-\theta_3^2}{4q^2(1+\theta_1)}w_0^{-\frac{1+\theta_1}{2}}\\
    &+C_0\bigg(\frac{1-\theta_1}{2}w_0^{\frac{1+\theta_1}{2}}+\frac{(1-\theta_1+\theta_2)^2-\theta_3^2}{4q^2(1-\theta_1)}w_0^{-\frac{1-\theta_1}{2}}\bigg)\bigg]\\
    &\bigg/\bigg[w_0^{-\frac{1+\theta_1}{2}}+C_0w_0^{-\frac{1-\theta_1}{2}}\bigg]\sim\frac{1+\theta_1}{2}w_0-C_0\theta_1w_0^{1+\theta_1}
\end{align*}
Then, for large $t$,
\begin{align*}
    E(t)&=\frac{-z_0}{\mu(1-\alpha)^2}\bigg\{\bigg[\frac{1+\theta_1}{2}z_0^{-\frac{1-\theta_1}{2}}F_0^+(z_0)+\frac{(1+\theta_1+\theta_2)^2-\theta_3^2}{4q^2(1+\theta_1)}z_0^{\frac{1+\theta_1}{2}}F_1^+(z_0)\\
    &+C_0\bigg(\frac{1-\theta_1}{2}z_0^{-\frac{1+\theta_1}{2}}F_0^-(z_0)+\frac{(1-\theta_1+\theta_2)^2-\theta_3^2}{4q^2(1-\theta_1)}z_0^{\frac{1-\theta_1}{2}}F_1^-(z_0)\bigg)\bigg]\\
    &\bigg/\bigg[z_0^{\frac{1+\theta_1}{2}}F_0^+(z_0)+C_0z_0^{\frac{1-\theta_1}{2}}F_0^-(z_0)\bigg]+\frac{1}{2}\frac{1+\theta_2}{z_0-q^2}\\
    &-\frac{1}{2}\bigg[\frac{1}{z_0}-\frac{\mu}{z_0}\bigg(2\alpha(1-\alpha)\frac{z_0-p^2}{z_0-q^2}-1\bigg)\bigg]\bigg\}\\
    &\sim\frac{-z_0}{\mu(1-\alpha)^2}\bigg[\frac{1+\theta_1}{2}\frac{1}{z_0}-C_0\theta_1\frac{1}{z_0^{1+\theta_1}}-\frac{1}{2}\frac{1+\theta_2}{q^2}-\frac{1}{2z_0}+\frac{\mu}{2z_0}\bigg(2\alpha(1-\alpha)\frac{p^2}{q^2}-1\bigg)\bigg]\\
    &\sim\frac{-z_0}{\mu(1-\alpha)^2}\bigg[\frac{\theta_1}{2z_0}+\frac{\mu}{2z_0}\bigg(2\alpha(1-\alpha)\frac{p^2}{q^2}-1\bigg)-C_0\theta_1\frac{1}{z_0^{1+\theta_1}}\bigg]\\
    &=\frac{1}{2(1-\alpha)^2}\bigg(1-\frac{2\alpha(1-\alpha)p^2}{q^2}-\frac{\sqrt{q^2-4\alpha(1-\alpha)p^2}}{q}\bigg)+\frac{C_0\theta_1p^2}{\mu(1-\alpha)^2}e^{-\lambda_B(p-q)\theta_1t}
\end{align*}
Again, it is straight-forward to show that the constant part is between $0$ and $1$ and, therefore, represents a valid probability of extinction.

\noindent The same approach works for the case $\theta_1<-1$. We can show that
\begin{align*}
    &\bigg[\frac{1+\theta_1}{2}w_0^{\frac{1-\theta_1}{2}}+\frac{(1+\theta_1+\theta_2)^2-\theta_3^2}{4q^2(1+\theta_1)}w_0^{-\frac{1+\theta_1}{2}}\\
    &+C_0\bigg(\frac{1-\theta_1}{2}w_0^{\frac{1+\theta_1}{2}}+\frac{(1-\theta_1+\theta_2)^2-\theta_3^2}{4q^2(1-\theta_1)}w_0^{-\frac{1-\theta_1}{2}}\bigg)\bigg]\\
    &\bigg/\bigg[w_0^{-\frac{1+\theta_1}{2}}+C_0w_0^{-\frac{1-\theta_1}{2}}\bigg]\sim\frac{1+\theta_1}{2}w_0+\frac{(1+\theta_1+\theta_2)^2-\theta_3^2}{4q^2(1+\theta_1)}
\end{align*}
Then, the probability of extinction in large time is approximated by
\begin{align*}
    E(t)&\sim\frac{1}{2(1-\alpha)^2}\bigg(1-\frac{2\alpha(1-\alpha)p^2}{q^2}-\frac{\sqrt{q^2-4\alpha(1-\alpha)p^2}}{q}\bigg)\\
    &-\frac{p^2}{\mu(1-\alpha)^2}\bigg[\frac{(1+\theta_1+\theta_2)^2-\theta_3^2}{4q^2(1+\theta_1)}-\frac{1+\theta_2}{2q^2}\bigg]e^{-\lambda_B(p-q)\theta_1t}
\end{align*}
The theorem then follows. 

\end{document}